\documentclass[ECP,preprint]{ejpecp}

\SHORTTITLE{Distance two for Poisson--binomial point probabilities}
\TITLE{Why Distance Two Is Exceptional for Poisson--Binomial Point Probabilities}
\AUTHORS{Igor~Kleiner\thanks{Technion -- Israel Institute of Technology, Haifa, Israel; Holon Institute of Technology (HIT), Holon, Israel. \EMAIL{igkleiner@gmail.com}. ORCID: \href{https://orcid.org/0000-0002-8361-8505}{0000-0002-8361-8505}.}}
\AMSSUBJ{60C05}
\AMSSUBJSECONDARY{60E15 ; 26C10}
\SUBMITTED{}
\ACCEPTED{}
\VOLUME{0}
\YEAR{2026}
\PAPERNUM{0}
\DOI{}

\ABSTRACT{Let $S$ be a finite sum of independent Bernoulli random variables.  We ask
how large the mean must be before the atom at $k+r$ can strictly exceed the
atom at $k$.  For the threshold
$T_{k,r}:=\inf\{\E S:\Pp(S=k+r)>\Pp(S=k)\}$, the adjacent case recovers the
classical Darroch boundary $k+1/(k+2)$, while every gap $r\ge3$ has threshold
$k+1$.  Distance two is the unique nonadjacent exception: $T_{k,2}=k+\Delta_k$,
where $\Delta_k$ and its staircase optimizer are explicit and
$\Delta_k=1/2+1/\sqrt{2k}-1/(4k)+O(k^{-3/2})$.  The proof combines the
classical fixed-mean reduction to shifted binomial laws with an exact discrete
optimization.  The result also yields sharp restrictions on nonadjacent
coefficient reversals in real-rooted polynomials with nonnegative
coefficients.}

\newcommand{\Pp}{\mathbb P}
\newcommand{\E}{\mathbb E}
\newcommand{\R}{\mathbb R}

\begin{document}

\section{Introduction}

Knowing the mean of a count is much easier than knowing its whole distribution.
But what does the mean tell us about two individual probabilities?  Let
\[
S=X_1+\cdots+X_n,\qquad X_i\sim\operatorname{Bernoulli}(p_i),
\]
where the summands are independent and $0\le p_i\le1$.  Fix integers $k\ge0$
and $r\ge1$.  We ask how large $\E S$ must be before $\Pp(S=k+r)$ can exceed
$\Pp(S=k)$.  More precisely, define
\begin{equation}\label{eq:defT}
T_{k,r}:=\inf\bigl\{\E S:\Pp(S=k+r)>\Pp(S=k)\bigr\},
\end{equation}
where the infimum ranges over all finite Bernoulli sums, with no prescribed
number of summands.  Probabilities outside the support are set to zero.
The strict comparison excludes ties between two zero masses.

The adjacent case belongs to the classical theory of the mode.  Darroch
\cite[Theorem~4]{Darroch1964} proved that every mode of a Bernoulli sum lies
within one unit of its mean and obtained the sharper boundary $k+1/(k+2)$ for
a change between $k$ and $k+1$.  Samuels \cite{Samuels1965} studied the
successive differences of these distributions.  Pitman \cite{Pitman1997}
presents these results together with their real-rooted polynomial
interpretation; his bounds for consecutive coefficient ratios are
particularly close to our question.  Baker and Handelman
\cite{BakerHandelman1992} use the ratio of the second-largest to the largest
mass of a unimodal law as a measure of local flatness.  More recently, Gnedin
\cite{Gnedin2024} studied how cheaply an added Bernoulli or Poisson increment
can move the mode.

Distance two is linked to this circle of ideas by a one-line identity.  If
$C\sim\operatorname{Bernoulli}(1/2)$ is independent of $S$, then
\[
\Pp(S+C=k+2)-\Pp(S+C=k+1)
=\frac{\Pp(S=k+2)-\Pp(S=k)}2.
\]
Thus the case $r=2$ asks for the smallest mean of a Bernoulli sum for which a
fixed fair coin is critical for moving the mode past $k+1$.  Gnedin fixes the
base law and optimizes the increment; here the increment is fixed and the base
Bernoulli sum is optimized under a mean constraint.

For a nonadjacent pair, a reversal can happen in two ways: the two positive
masses can cross, or the support can move past the lower index.  These
mechanisms behave differently.  At distance two and $k\ge1$, the crossing
comes first and leaves a genuine discrete optimization; at $k=0$ the two
boundaries coincide.  From distance three onward, the support barrier $k+1$
wins.  We evaluate the crossing level exactly and describe all ties at the
boundary.

The first step is classical.  The extremal idea goes back to Tchebichef
\cite{Tchebichef1846}; Hoeffding's general fixed-mean result
\cite[Theorem~3]{Hoeffding1956} reduces the symmetric multi-affine functional
used here to parameters taking only the values $0$, $1$, and one common
interior value.  Mattner \cite{Mattner2022} gives a modern multi-moment
extension.  The one-moment argument is short enough to include, after which
the problem becomes an explicit discrete optimization.

For $k\ge0$ put
\begin{equation}\label{eq:dk}
 d_k:=\max\left\{2,\left\lfloor\frac{9+\sqrt{8k+33}}8\right\rfloor\right\},
\end{equation}
and define
\begin{equation}\label{eq:Deltak}
\Delta_k:=
\frac{d_k\sqrt{(k+1)(k+2)}-k\sqrt{d_k(d_k-1)}}
{\sqrt{(k+1)(k+2)}+\sqrt{d_k(d_k-1)}}.
\end{equation}

\begin{theorem}[Sharp reversal thresholds]\label{thm:main}
For every $k\ge0$ and $r\ge1$,
\[
T_{k,r}=\begin{cases}
\displaystyle k+\frac1{k+2},&r=1,\\[2mm]
\displaystyle k+\Delta_k,&r=2,\\[2mm]
 k+1,&r\ge3.
\end{cases}
\]
For $r=2$ the boundary is realized by a binomial distribution:
if
\[
q_k:=\frac{\sqrt{(k+1)(k+2)}}
{\sqrt{(k+1)(k+2)}+\sqrt{d_k(d_k-1)}},
\]
then for $B_k\sim\operatorname{Bin}(k+d_k,q_k)$,
\[
\E B_k=k+\Delta_k,
\qquad
\Pp(B_k=k)=\Pp(B_k=k+2).
\]
Increasing $q_k$ gives strict reversal.  In every case,
$\Pp(S=k+r)>\Pp(S=k)$ implies $\E S>T_{k,r}$;
the infimum in \eqref{eq:defT} is not attained by a strict reversal.
\end{theorem}

The middle line is the part to understand.  One might expect the boundary to
be reached by $k+2$ trials, as in the adjacent case.  It is not: $d_k=2$ only
for $k\le23$, and then the optimizer grows as $d_k\sim\sqrt{k/8}$.

\begin{remark}[Nondegenerate trials]\label{rem:nondegenerate}
The values in Theorem~\ref{thm:main} are unchanged if the infimum is restricted
to sums for which every Bernoulli parameter satisfies $0<p_i<1$.  For
$r=1,2$ the boundary binomial examples are already nondegenerate.  For
$r\ge3$ an explicit perturbation inside the open cube is given in
Section~\ref{sec:gaps}.
\end{remark}

The exceptional constant has a simple asymptotic form.

\begin{corollary}\label{cor:asymptotic}
As $k\to\infty$,
\[
 d_k=\sqrt{k/8}+O(1)
\]
and
\[
\Delta_k
=\frac12+\frac1{\sqrt{2k}}-\frac1{4k}+O(k^{-3/2}).
\]
\end{corollary}

A final corollary recasts the result for real-rooted polynomials.  If
$A(z)=\sum a_jz^j$ has nonnegative coefficients and only real zeros, then the
normalized coefficients form a Poisson--binomial law and $A'(1)/A(1)$ is its
mean.  Thus Theorem~\ref{thm:main} also gives the sharp
first-logarithmic-derivative threshold forced by a nonadjacent coefficient
reversal.

\section{An extremal reduction}

The following elementary lemma is the one-moment case of the
Chebyshev--Hoeffding extremal principle; compare \cite{Hoeffding1956} and the
multi-moment extension in \cite{Mattner2022}.  We record the short pairwise
argument because the same deformation will reappear when we analyze boundary
ties.

\begin{lemma}[Shifted-binomial reduction]\label{lem:reduction}
Let $F:[0,1]^n\to\R$ be symmetric and affine in each coordinate separately.
Fix $\mu\in[0,n]$.  A maximum of $F$ on
\[
\{(p_1,\ldots,p_n)\in[0,1]^n:\ p_1+\cdots+p_n=\mu\}
\]
can be attained at a vector whose coordinates belong to $\{0,q,1\}$ for a
single $q\in(0,1)$.
\end{lemma}

\begin{proof}
Choose a maximizer with as few coordinates in $(0,1)$ as possible.  Freeze all
coordinates except two interior coordinates $x,y$.  Symmetry and separate
affine-linearity give
\[
F(x,y)=\alpha xy+\beta(x+y)+\gamma.
\]
On the feasible segment $x+y=c$ this becomes
\[
\alpha x(c-x)+\beta c+\gamma.
\]
If $\alpha=0$, the value is constant along the segment and one of $x,y$ can
be pushed to an endpoint, contradicting minimality.  If $\alpha<0$, the
restriction is convex and cannot have an interior maximum.  Thus, whenever two
interior coordinates remain, $\alpha>0$ and the interior maximum occurs at
$x=y=c/2$.
Applying this to every pair shows that all interior coordinates are equal.
\end{proof}

\begin{lemma}[A strict reversal cannot attain the infimum]\label{lem:open}
If $\Pp(S=k+r)>\Pp(S=k)$, then $\E S>T_{k,r}$.
\end{lemma}
\begin{proof}
A strict reversal has $\mu:=\E S>0$.  Replace every $p_i$ by $\lambda p_i$, where $\lambda<1$ is sufficiently close
to $1$.  Point probabilities are continuous polynomials in the parameters, so
the reversal persists, whereas the mean becomes $\lambda\mu<\mu$.  Hence
$T_{k,r}\le\lambda\mu<\mu$.
\end{proof}

For fixed $k,r$, the functional
\[
F(p_1,\ldots,p_n)=\Pp(S=k+r)-\Pp(S=k)
\]
is symmetric and multi-affine.  Therefore, if a strict reversal is possible at
a given mean, it is possible for a law of the form
\begin{equation}\label{eq:shiftbin}
S=\ell+\operatorname{Bin}(m,q).
\end{equation}
Thus the fixed-mean comparison reduces to shifted binomial laws.

\section{After the reduction: the boundary formula}

Take $m\ge1$ and $0<q<1$ in \eqref{eq:shiftbin}.  Purely deterministic
laws are considered separately below.  Assume first that $\ell\le k$, and put
\[
s:=k-\ell,\qquad d:=m-s.
\]
For a strict reversal we need $m\ge s+r$, or $d\ge r$; otherwise
$\Pp(S=k+r)=0$.  Under this condition $\Pp(S=k)>0$, so division by that
probability is legitimate.  Write
\[
t:=\frac{q}{1-q}.
\]
Then
\begin{equation}\label{eq:ratio}
\frac{\Pp(S=k+r)}{\Pp(S=k)}
=
\frac{\binom{m}{s+r}}{\binom{m}{s}}\,t^r.
\end{equation}
Hence equality of the two atoms occurs at
\begin{equation}\label{eq:rho}
 t=\rho_{s,d,r}
 :=\left(\prod_{j=1}^r\frac{s+j}{d-j+1}\right)^{1/r}.
\end{equation}
Since $m=s+d$ and $\ell=k-s$,
\begin{equation}\label{eq:meanexcess}
\E S-k
=\frac{dt-s}{1+t},
\end{equation}
which is strictly increasing in $t$.  Thus the mean excess at the boundary is
\begin{equation}\label{eq:delta-master}
\delta_{s,d,r}
:=\frac{d\rho_{s,d,r}-s}{1+\rho_{s,d,r}}.
\end{equation}

There is one other way to obtain a reversal: move the support past $k$.  If
$\ell\ge k+1$, then $\E S\ge k+1$.  Conversely, for $r=1$ the
constant $S=k+1$ already gives a strict reversal.  For $r\ge2$, the law
$(k+1)+\operatorname{Bin}(r-1,q)$ gives a strict reversal and has mean
decreasing to $k+1$ as $q\downarrow0$.  Deterministic laws cannot give a
smaller reversal mean.  Together with Lemma~\ref{lem:reduction}, this proves
\begin{equation}\label{eq:master-inf}
T_{k,r}=k+
\min\left\{1,
\inf_{\substack{0\le s\le k\\ d\ge r}}
\delta_{s,d,r}
\right\}.
\end{equation}
Here and below the indices $s,d$ in an infimum are integers.  For each such
pair, raising $t$ slightly above $\rho_{s,d,r}$ gives a strict reversal.
Thus \eqref{eq:master-inf} supplies both the lower bound and sharpness.  From
here the problem is concrete: minimize the positive-mass crossing level and
compare it with the support barrier $1$.

\section{Adjacent atoms and gaps at least three}\label{sec:gaps}

For $r=1$,
\[
\rho_{s,d,1}=\frac{s+1}{d},
\qquad
\delta_{s,d,1}=\frac{d}{s+d+1}.
\]
The minimum is at $s=k,d=1$, so
\[
T_{k,1}=k+\frac1{k+2}.
\]
This is Darroch's boundary \cite[Theorem~4]{Darroch1964}; the corresponding
extremal vector is also recorded by Gnedin \cite[Sec.~5.1]{Gnedin2024}.  We
include the adjacent case only to place all three gap regimes in one formula.

Now suppose $r\ge3$.  Here the product structure is already enough to settle
the comparison.  From \eqref{eq:rho},
\[
\rho_{s,d,r}^r
=\frac{(s+1)(s+2)\cdots(s+r)}{d(d-1)\cdots(d-r+1)}.
\]
The numerator is strictly greater than $(s+1)^r$.  For the denominator,
\[
d(d-2)\le(d-1)^2,
\]
and all remaining factors are at most $d-1$.  Hence
\[
d(d-1)\cdots(d-r+1)\le(d-1)^r,
\]
so
\begin{equation}\label{eq:rho-lower}
\rho_{s,d,r}>\frac{s+1}{d-1}.
\end{equation}
The map $\rho\mapsto(d\rho-s)/(1+\rho)$ is increasing.  Substituting the
right-hand side of \eqref{eq:rho-lower} gives exactly $1$.  Consequently
\[
\delta_{s,d,r}>1,
\]
and \eqref{eq:master-inf} yields $T_{k,r}\ge k+1$.

For the reverse inequality, fix $0<\varepsilon<r-1$ and take
\[
S=(k+1)+\operatorname{Bin}\left(r-1,\frac{\varepsilon}{r-1}\right).
\]
Then $\E S=k+1+\varepsilon$, $\Pp(S=k)=0$, and
$\Pp(S=k+r)>0$.  Therefore
\[
T_{k,r}=k+1\qquad(r\ge3).
\]

This conclusion does not rely on deterministic trials.  Let $q\downarrow0$ and
put
\[
\delta:=\frac{q^{r-1}}{4(k+1)},
\qquad
S_q\sim\operatorname{Bin}(k+1,1-\delta)
      +\operatorname{Bin}(r-1,q),
\]
with the two binomials independent.  Every Bernoulli parameter now lies in
$(0,1)$ and
\[
\E S_q=k+1+(r-1)q-(k+1)\delta\longrightarrow k+1
\]
from above.  Moreover,
\[
\Pp(S_q=k+r)=q^{r-1}(1+o(1)),
\qquad
\Pp(S_q=k)=\frac14 q^{r-1}(1+o(1)),
\]
so the strict reversal holds for all sufficiently small $q$.  This proves the
claim in Remark~\ref{rem:nondegenerate}.

\section{The exceptional gap \texorpdfstring{$r=2$}{r=2}}

For $r=2$ write
\[
A_s:=\sqrt{(s+1)(s+2)},
\qquad
H_d:=\sqrt{d(d-1)}.
\]
At equality of the two atoms,
\[
t=\frac{A_s}{H_d},
\]
and \eqref{eq:delta-master} becomes
\begin{equation}\label{eq:delta2sd}
\delta_{s,d,2}
=\frac{dA_s-sH_d}{A_s+H_d}.
\end{equation}

\subsection{The optimal deterministic shift is zero}

First observe from \eqref{eq:delta2sd} that
\begin{equation}\label{eq:delta-less-one}
\delta_{s,d,2}<1
\quad\Longleftrightarrow\quad
 d<s+2.
\end{equation}
For $k\ge1$, the candidate $s=k,d=2$ already gives a value below $1$.
Thus any minimizer relevant to \eqref{eq:master-inf} must satisfy
$d\le s+1$.

The point of the next calculation is to show that no deterministic shift
can improve the crossing level.  Fix $d$ and extend $s$ continuously.  With
$A=A_s$ and $H=H_d$,
\[
\frac{\partial}{\partial s}\delta_{s,d,2}
=
\frac{H\bigl((d+s)A'-A-H\bigr)}{(A+H)^2},
\qquad
A'=\frac{2s+3}{2A}.
\]
After multiplication by the positive factor $2A$, the sign is the sign of
\[
2sd+3d-3s-4-2AH.
\]
Since $A>s+1$ and $H>d-1$,
\[
2AH>2(s+1)(d-1),
\]
while
\[
2(s+1)(d-1)-(2sd+3d-3s-4)=s+2-d>0
\]
under $d\le s+1$.  Hence the derivative is negative.  The smallest boundary
mean is therefore obtained at the largest possible $s$, namely $s=k$.  The
case $k=0$ has only $s=0$: \eqref{eq:delta-less-one} gives
$\delta_{0,d,2}\ge1$, with equality at $d=2$.  There is also a direct way to
see this boundary.  For any nonnegative integer-valued $Y$ with $\E Y=1$,
\[
\Pp(Y=2)-\Pp(Y=0)
=-\sum_{j\ge3}(j-1)\Pp(Y=j)\le0.
\]
Thus $T_{0,2}=1$ without using the Bernoulli structure.  The support term $1$
in \eqref{eq:master-inf} therefore does not alter the minimum below.

For $k\ge1$, any candidate below the support threshold satisfies
$2\le d\le s+1\le k+1$, so the relevant set is finite; the infimum is a
minimum.  It follows that
\begin{equation}\label{eq:phidef}
T_{k,2}=k+\min_{d\ge2}\phi_k(d),
\qquad
\phi_k(d):=
\frac{dA_k-kH_d}{A_k+H_d},
\end{equation}
where $A_k=\sqrt{(k+1)(k+2)}$.

\subsection{The exact integer optimizer}

The remaining minimization has an exact staircase structure.

\begin{lemma}\label{lem:staircase}
For $d\ge2$ define
\[
K_d:=8d^2-2d-4.
\]
Then, for every integer $k\ge0$,
\[
\phi_k(d)<\phi_k(d+1)\quad\Longleftrightarrow\quad k<K_d,
\]
and
\[
\phi_k(d)>\phi_k(d+1)\quad\Longleftrightarrow\quad k\ge K_d.
\]
There is no integer tie.
\end{lemma}

The proof is given in Appendix~\ref{app:staircase}.  Since $K_d$ is strictly
increasing, Lemma~\ref{lem:staircase} implies that $\phi_k(d)$ decreases up to
a unique integer and then increases.  Solving
$K_{d-1}\le k<K_d$ gives exactly \eqref{eq:dk}.  Substituting $d=d_k$ in
\eqref{eq:phidef} proves \eqref{eq:Deltak} and completes the proof of
Theorem~\ref{thm:main}.

\section{What happens at the boundary}\label{sec:boundary}

The infimum concerns strict reversal, but the ties at its boundary explain the
extremizers.  They also reveal a small exception at $k=0$.  The reduction
lemma produces a reduced maximizer, not a classification of every equality
case, so we need one final rigidity argument.

\begin{proposition}[Boundary ties]\label{prop:ties}
Suppose $\E S=T_{k,r}$.  Then $\Pp(S=k+r)\le\Pp(S=k)$.
After zero-parameter trials are removed, equality holds precisely in the
following cases:
\begin{enumerate}
\item For $r=1$, the law is
$\operatorname{Bin}(k+1,(k+1)/(k+2))$.
\item For $r=2$ and $k\ge1$, the law is the binomial $B_k$ of
Theorem~\ref{thm:main}.
\item For $r=2$ and $k=0$, at most two Bernoulli parameters are positive,
and their sum is $1$.
\item For $r\ge3$, $S=k+1$ almost surely; both compared masses are zero.
\end{enumerate}
\end{proposition}

\begin{proof}
The non-strict inequality follows from Lemma~\ref{lem:open}.  First consider
$r=2,k=0$.  The mean identity at $\E S=1$ gives
\[
 \Pp(S=2)-\Pp(S=0)=-\sum_{j\ge3}(j-1)\Pp(S=j).
\]
Thus equality is equivalent to $S\le2$ almost surely, which for a Bernoulli
sum is equivalent to having at most two positive parameters.  In particular,
$\operatorname{Bernoulli}(p)+\operatorname{Bernoulli}(1-p)$ is a boundary
tie for every $p\in[0,1]$, not only for $p=1/2$.

For the other cases, set $F=\Pp(S=k+r)-\Pp(S=k)$ on the fixed-mean slice.
At the threshold the maximum of $F$ is zero, so every equality point is a
maximizer.  If two interior
parameters $x\ne y$ remain, write $F=\alpha xy+\beta(x+y)+\gamma$ as in the reduction lemma.
Stationarity along $x+y=c$ gives $\alpha(y-x)=0$, hence $\alpha=0$.  We can therefore
move the pair to an endpoint without changing $F$, reducing the number of
interior parameters.  Repetition ends at a shifted-binomial maximizer.

For $r=1$, the boundary calculation has the unique minimizing pair
$s=k,d=1$.  For $r=2,k\ge1$, strict decrease in the shift and the staircase
lemma give the unique pair $s=k,d=d_k$.  The mean is in $(k,k+1)$ in both
cases, so the support contains $k$ and a zero-mass tie is impossible.  Thus
the endpoint of the reduction is the indicated binomial, with no parameter
$1$.

It remains to rule out equality points hidden by flat steps of the
reduction.  Suppose at least one such step was used.  Its last nontrivial step
must end with a pair $(0,q)$, because the terminal binomial has no parameter
equal to $1$.  Reverse the step by replacing $(0,q)$ with
$(\varepsilon,q-\varepsilon)$.  If $b_j$ denotes the mass at $j$ of the
remaining Bernoulli sum, then the varying pair has generating factor
\[
(1-\varepsilon+\varepsilon z)(1-q+\varepsilon+(q-\varepsilon)z)
=(1-q+qz)+(q\varepsilon-\varepsilon^2)(z-1)^2.
\]
Hence the directional derivative of $F$ at $\varepsilon=0$ is
\[
D=q\{b_{k+r}-2b_{k+r-1}+b_{k+r-2}-b_k+2b_{k-1}-b_{k-2}\},
\]
with missing-index masses interpreted as zero.  Let $c$ be the terminal
binomial mass at $k$.  For $r=1$ this gives
\[
 \frac Dc=-\frac{(k+2)(k+3)}{2(k+1)^2}<0.
\]
For $r=2$, write $d=d_k$, $t=\sqrt{(k+1)(k+2)/(d(d-1))}$ and
$q=t/(1+t)$.  Then
\begin{align*}
 \frac Dc
 &=\frac{dt}{k+1}-\frac{k}{(d+1)t}-\frac{2(1+t)^2}{(k+d)t}\\
 &=-\frac{2\{t^2+2(d+1)t-(2k+1)\}}{(d+1)(k+d)t}<0.
\end{align*}
The sign follows from $(d+1)t>k+1$.  A flat step would have derivative zero,
so in either case we have a contradiction.  The starting vector was already
the stated binomial vector.

For $r\ge3$, \eqref{eq:rho-lower} is strict, so no shifted binomial with
both compared masses positive can tie at mean $k+1$.  A shifted-binomial
zero-mass tie at this mean must be the constant $k+1$.  A last reduction
step ending there would replace interior parameters $x,1-x$ by $0,1$.
Before that step the law is $k+X+Y$ with $0<x<1$, giving
$F=-x(1-x)<0$, again a contradiction.  Thus the original law is constant.
\end{proof}

\section{Asymptotics}

The staircase itself already reveals the relevant scale:
\[
d_k=\sqrt{k/8}+O(1).
\]
Write $d_k=\sqrt{k/8}+b_k$, where $(b_k)$ is bounded.  Expanding the two
radicals in \eqref{eq:Deltak}, uniformly for bounded $b_k$, gives
\[
\Delta_k
=\frac12+\frac1{\sqrt{2k}}-\frac1{4k}+O(k^{-3/2}),
\]
which proves Corollary~\ref{cor:asymptotic}.

\section{Real-rooted polynomials}

Let
\[
A(z)=\sum_{j=0}^n a_jz^j,
\qquad a_j\ge0,
\qquad A(1)>0,
\]
and suppose that all zeros of $A$ are real.  Then they are nonpositive, and
after normalization the coefficient sequence is the law of a sum of
independent Bernoulli variables; see \cite{Pitman1997}.  In particular,
\[
\mu:=\frac{A'(1)}{A(1)}
\]
is the mean of that Bernoulli sum.

\begin{corollary}[Sharp nonadjacent coefficient reversal]\label{cor:poly}
If $a_{k+r}>a_k$, then
\[
\frac{A'(1)}{A(1)}>T_{k,r},
\]
where $T_{k,r}$ is given in Theorem~\ref{thm:main}.  Every constant in this
statement is sharp.
\end{corollary}

For $r=2$, equality at the boundary is furnished, up to an irrelevant positive
constant, by
\[
A(z)=(1-q_k+q_k z)^{k+d_k},
\]
for which $a_{k+2}=a_k$ and $A'(1)/A(1)=k+\Delta_k$.

\section{Discussion}

The fair-coin identity puts the exceptional case directly into the language of
mode transport: the distance-two comparison is an adjacent comparison with one
success probability fixed at $1/2$.  Baker--Handelman's ratio \cite{BakerHandelman1992} records how flat a unimodal
peak is, while Gnedin's peak skewness \cite{Gnedin2024} records the smallest
Bernoulli parameter that moves the leading mode of a fixed base law.  Here the perturbing parameter is
fixed at $1/2$, and the quantity being minimized is instead the mean of the
base Bernoulli sum.  At the distance-two boundary the optimizing law has unique
mode $k+1$ and equal flanks at $k$ and $k+2$.  Thus a fair Bernoulli is exactly
critical for moving that mode: after convolution, the masses at $k+1$ and
$k+2$ tie.  In Gnedin's terminology \cite{Gnedin2024}, the critical peak skewness is $1/2$.
Applying the unrestricted refined Darroch bound to $S+C$ gives only
\[
 \Pp(S=k+2)>\Pp(S=k)
 \quad\Longrightarrow\quad
 \E S>k+\frac12+\frac1{k+3}.
\]
The exact value $k+\Delta_k$ is the cost of retaining that fixed-coin
constraint.

More generally, an independent uniform variable $U$ on
$\{0,\ldots,r-1\}$ satisfies
\[
 \Pp(S+U=k+r)-\Pp(S+U=k+r-1)
 =\frac{\Pp(S=k+r)-\Pp(S=k)}r.
\]
For $r=2$, adding $U$ stays within the class of Bernoulli sums.  For
$r\ge3$, its generating polynomial $1+z+\cdots+z^{r-1}$ has nonreal zeros,
so this closure is lost.  The exact product estimate in Section~\ref{sec:gaps}
is what settles those larger gaps.

The same product also explains why two is the exceptional gap.  For $r\ge3$
the leading factor $d$ in $d(d-1)\cdots(d-r+1)$ can be paired with $d-2$ and
absorbed by $(d-1)^2$ through $d(d-2)\le(d-1)^2$.  At $r=2$ there is no
$d-2$ factor, so this one-line comparison is unavailable and the discrete
optimization survives.

The staircase starts with $d_k=2$ for $0\le k\le23$, $d_k=3$ for
$24\le k\le61$, and $d_k=4$ for $62\le k\le115$.
Its scale $d_k\sim\sqrt{k/8}$ explains the $k^{-1/2}$ correction to $1/2$.
At the boundary, the optimizing binomial is unique for $k\ge1$; the family
of complementary pairs at $k=0$ is the only exception for distance two.

\begin{remark}[A Poisson component]
The thresholds also extend to $S=Z+\sum_{i=1}^n X_i$, where
$Z\sim\operatorname{Poi}(\lambda)$ is independent and $\lambda\ge0$.
Replace $Z$ by $\operatorname{Bin}(N,\lambda/N)$ with $N>\lambda$.
The mean is unchanged, point probabilities converge, and a strict reversal
persists for all sufficiently large $N$.  The finite-sum bound therefore
still applies, including its strictness.  Conversely, the finite binomial
examples already establish sharpness.  Even when $\lambda>0$ is required,
adding an arbitrarily small Poisson component to a strict finite example
preserves the reversal by continuity.  This extends the thresholds, not the
equality classification in Proposition~\ref{prop:ties}.
\end{remark}

\begin{remark}[A fixed number of trials]
If the number $n$ of Bernoulli trials is prescribed and zero parameters are
allowed, the same thresholds are already available for $n\ge k+1$ when
$r=1$, $n\ge k+d_k$ when $r=2$, and $n\ge k+r$ when $r\ge3$.  Below these
values the constrained threshold can only increase (and can be infinite when
the upper atom lies outside the support).
\end{remark}

The original question therefore has a complete explicit answer.  A first-moment
constraint controls the order of two specified atoms, with a binomial crossing
at distance two and a support threshold at every larger distance.  The
exceptional constant records exactly where those two mechanisms part company.

\appendix
\section{Proof of the staircase lemma}\label{app:staircase}

Set
\[
A:=\sqrt{(k+1)(k+2)},
\qquad
H_d:=\sqrt{d(d-1)}.
\]
A common-denominator calculation gives
\begin{align*}
\phi_k(d)-\phi_k(d+1)
&=
\frac{A}{(A+H_d)(A+H_{d+1})}\\
&\quad\times
\left[(k+d)H_{d+1}-(k+d+1)H_d-A\right].
\end{align*}
Thus the sign is the sign of $R_d(k)-A$, where
\[
R_d(k):=(k+d)H_{d+1}-(k+d+1)H_d.
\]
Put
\[
\eta:=d-\sqrt{d^2-1}.
\]
For $d\ge2$ one has $0<\eta<1/3$.  If $u:=H_{d+1}-H_d$, squaring and simplifying gives
\[
u^2=1+\eta^2,
\qquad
 d-\frac{H_d}{u}=\frac{1+\eta}{2}.
\]
Consequently
\[
R_d(k)=u\left(k+\frac{1+\eta}{2}\right).
\]
Because both sides are positive, the sign of $R_d(k)-A$ equals the sign of
\[
F_d(k):=
(1+\eta^2)\left(k+\frac{1+\eta}{2}\right)^2-(k+1)(k+2).
\]
This is a quadratic in $k$ with positive leading coefficient and negative
constant term, hence it has exactly one positive zero.

Now set $K_d=8d^2-2d-4$.  Using $2d=\eta+\eta^{-1}$ and simplifying gives
\begin{align*}
4F_d(K_d)
&=16\eta^6-8\eta^5+9\eta^4-10\eta^3
 +18\eta^2-2\eta+1\\
&=(1-2\eta)+\eta^2(18-10\eta)
 +\eta^4(9-8\eta)+16\eta^6>0,
\end{align*}
while
\begin{align*}
4F_d(K_d-1)
&=16\eta^6-8\eta^5-7\eta^4-6\eta^3
 +18\eta^2+2\eta-7\\
&<-7+2\eta+18\eta^2+16\eta^6<0,
\end{align*}
since $0<\eta<1/3$.  The unique positive zero of $F_d$ therefore lies
strictly between $K_d-1$ and $K_d$, which proves the lemma.

\end{document}